\documentclass[draft]{amsart}

\usepackage{enumerate}
\usepackage{amsmath, amssymb, epsfig, graphicx, url}
\usepackage{color}

\newtheorem{thm}{Theorem}[section]
\newtheorem{lem}[thm]{Lemma}
\newtheorem{cor}[thm]{Corollary}

\newtheorem{exam}[thm]{Example}
\newtheorem{defn}[thm]{Definition}
\newtheorem{problem}{Problem}

\renewcommand{\leq}{\leqslant}
\renewcommand{\geq}{\geqslant}

\renewcommand{\ker}{{\rm ker}}

\newcommand{\rank}{\operatorname{rank}}

\newcommand{\genset}[1]{\ensuremath{\langle\: #1 \:\rangle}}
\newcommand{\trans}{\mathcal{T}_{n}}
\newcommand{\sym}{\mathcal{S}_{n}}
\newcommand{\alt}{\mathcal{A}_{n}}

\newcommand{\agl}{\mbox{\rm AGL}}
\newcommand{\psl}{\mbox{\rm PSL}}
\newcommand{\pgl}{\mbox{\rm PGL}}
\newcommand{\asl}{\mbox{\rm ASL}}
\newcommand{\pgaml}{\mbox{\rm P}\Gamma {\rm L}}
\newcommand{\agaml}{\mbox{\rm A}\Gamma {\rm L}}

\newcommand{\dihed}[1]{\ensuremath{D_{#1}}}

\begin{document}



\title[Groups and Transformation Semigroups]{The Classification of Normalizing Groups}
\author{Jo\~{a}o Ara\'{u}jo}
\author{Peter J. Cameron}
\author{James Mitchell}
\author{Max Neunh\"offer}

\address[Ara\'{u}jo]
{Universidade Aberta
and
Centro de \'{A}lgebra \\
Universidade de Lisboa \\
Av. Gama Pinto, 2, 1649-003 Lisboa \\ Portugal}
\email{\url{jaraujo@ptmat.fc.ul.pt}}

\address[Cameron]{Department of Mathematics \\
School of Mathematical Sciences at Queen Mary\\ University of London}
\email{\url{P.J.Cameron@qmul.ac.uk }}

\address[Mitchell]{Mathematical Institute,
University of St Andrews, North Haugh, St Andrews, Fife, KY16 9SS, Scotland\\}
\email{\url{jamesm@mcs.st-and.ac.uk}}

\address[Neunh\"offer]{Mathematical Institute,
University of St Andrews, North Haugh, St Andrews, Fife, KY16 9SS, Scotland}
\email{\url{neunhoef@mcs.st-and.ac.uk}}


\begin{abstract}
Let $X$ be a finite set such that $|X|=n$. Let $\trans$ and $\sym$ denote the transformation monoid and the symmetric group on $n$ points, respectively. Given $a\in \trans\setminus \sym$, we say that a group $G\leq \sym$ is \emph{$a$-normalizing} if $$\langle a,G\rangle \setminus G=\langle g^{{-1}}ag\mid g\in G\rangle,$$
where $\langle a, G\rangle$ and $\langle g^{{-1}}ag\mid g\in G\rangle$ denote the subsemigroups of $\trans$ generated by the sets $\{a\}\cup G$ and $\{g^{-1}ag  \mid g\in G\}$, respectively. 
 If $G$ is $a$-normalizing for all $a\in \trans\setminus \sym$, then we say that $G$ is \emph{normalizing}.

The goal of this paper is to classify the normalizing groups and hence answer a question of Levi, McAlister, and McFadden. The paper ends with a number of problems for experts in groups, semigroups and matrix theory.

\end{abstract}

\maketitle
\medskip

\noindent{\em Date:} 10 August 2011\\
{\em Key words and phrases:} Transformation semigroups, permutation groups, primitive
groups, GAP\\
{\it 2010 Mathematics Subject Classification:}  20B30, 20B35,
20B15, 20B40, 20M20, 20M17. \\
{\em Corresponding author: Jo\~{a}o Ara\'{u}jo}





\newpage

\section{Introduction and Preliminaries}

For notation and basic results on group theory we refer the reader to \cite{cam,dixon}; for semigroup theory we refer the reader to \cite{Ho95}.
 Let $\trans$ and $\sym$ denote the  monoid consisting of mappings from $[n]:=\{1,\ldots ,n\}$ to $[n]$ and the symmetric group on $[n]$ points, respectively.  The monoid $\trans$ is usually called the full transformation semigroup. In \cite{lm}, Levi and McFadden proved the following result.

\begin{thm}
Let $a\in \trans\setminus \sym$. Then
\begin{enumerate}
\item  $\langle g^{-1}ag\mid g\in \sym\rangle$ is idempotent generated;
\item $\langle g^{-1}ag\mid g\in \sym\rangle$ is regular.
\end{enumerate}
\end{thm}

Using a beautiful argument, McAlister \cite{mcalister} proved that the semigroups $\langle g^{-1}ag\mid g\in \sym\rangle$ and $\langle a, \sym\rangle\setminus \sym$ (for $a\in \trans\setminus \sym$) have exactly the same set of  idempotents; therefore, as $\langle g^{-1}ag\mid g\in \sym\rangle$ is idempotent generated, it follows that
$$
 \langle g^{-1}ag\mid g\in \sym\rangle=\langle a, \sym\rangle\setminus \sym.
$$

Later, Levi \cite{levi96} proved that $\langle g^{-1}ag\mid g\in \sym\rangle= \langle g^{-1}ag\mid g\in \alt\rangle$ (for $a\in \trans\setminus \sym$), and hence the three results above remain true when we replace $\sym$ by $\alt$.
The following list of problems naturally arises from these considerations.

\begin{enumerate}
\item  Classify the groups $G\leq\sym $ such that for all $a\in \trans\setminus\sym$ we have   that the semigroup $\langle g^{-1}ag\mid g\in G\rangle$ is idempotent generated.
\item Classify the groups $G\leq\sym $ such that for all $a\in \trans\setminus\sym$ we have   that the semigroup $\langle g^{-1}ag\mid g\in G\rangle$ is regular.
\item Classify the groups $G\leq\sym $ such that for all $a\in \trans\setminus\sym$ we have $$\langle a,G\rangle\setminus G = \langle g^{-1}ag\mid g\in G\rangle.$$
\end{enumerate}

The two first questions were solved in \cite{ArMiSc} as follows:

\begin{thm}
If $n\geq 1$ and $G$ is a subgroup of $\sym$, then the following are equivalent:
\begin{enumerate}
\item[(i)] The semigroup
$\genset{g^{{-1}}ag \mid g\in G}$ is idempotent generated
for all $a\in \trans\setminus\sym$.
\item[(ii)] One of the following is valid for $G$ and $n$:
\begin{enumerate}
\item[(a)] $n=5$ and $G$ is $\agl(1,5)$;
\item[(b)] $n=6$ and $G$ is  $\psl(2,5)$ or $\pgl(2,5)$;
\item[(c)] $G$ is $\alt$ or $\sym$.
\end{enumerate}
\end{enumerate}
\end{thm}

\begin{thm}\label{th2}
If $n\geq 1$ and $G$ is a subgroup of $\sym$,  then the following are equivalent:
\begin{enumerate}
    \item[\rm (i)] The semigroup $\genset{g^{{-1}}ag \mid g\in G}$ is regular for all $a\in \trans\setminus\sym$.
\item[(ii)] One of the following is valid for $G$ and $n$:
    \begin{enumerate}
    \item[\rm (a)] $n=5$ and $G$ is  $ C_5,\ \dihed{5},$ or $\agl(1,5)$;
    \item[\rm (b)] $n=6$ and $G$ is $ \psl(2,5)$ or $\pgl(2,5)$;
    \item[\rm (c)]  $n=7$ and $G$ is $\agl(1,7)$;
    \item[\rm (d)] $n=8$ and $G$ is $\pgl(2,7)$;
    \item[\rm (e)] $n=9$ and $G$ is  $\psl(2,8)$ or $\pgaml(2,8)$;
    \item[\rm (f)] $G$ is $\alt$ or $\sym$.
\end{enumerate}
\end{enumerate}
\end{thm}

These results leave us with the third problem. Given $a\in \trans\setminus \sym$, we say that a group $G\leq \sym$ is \emph{$a$-normalizing} if $$\langle a,G\rangle \setminus G=\langle g^{{-1}}ag\mid g\in G\rangle.$$ If $G$ is $a$-normalizing for all $a\in \trans\setminus \sym$, then we say that $G$ is \emph{normalizing}. Recall that the \emph{rank} of a transformation $f$ is just the number of points in its image; we denote this by $\rank(f)$. For a given $k$ such that $1\leq k< n$, we say that $G$ is $k$-normalizing if $G$ is $a$-normalizing  for all rank $k$ maps $a\in \trans \setminus \sym$. 

Levi, McAlister and McFadden  \cite[p.464]{lmm} ask for a classification of all pairs $(a,G)$ such that $G$ is $a$-normalizing, and   in \cite{ArMiSc}  is proposed the more tractable problem of classifying the normalizing groups. The aim of this paper is to provide such a classification.

\begin{thm}\label{main}
If $n\geq 1$ and $G$ is a subgroup of $\sym$,  then the following are equivalent:
\begin{enumerate}
    \item[\rm (i)] The group $G$ is normalizing, that is, for all $a\in \trans \setminus \sym$ we have $$\langle a,G\rangle \setminus G = \langle g^{{-1}}ag\mid g\in G\rangle;$$
\item[(ii)] One of the following is valid for $G$ and $n$:
    \begin{enumerate}
    \item $n=5$ and $G$ is $ \agl(1,5)$;
    \item $n=6$ and  $G$ is  $ \psl(2,5)$ or $\pgl(2,5)$;
    \item $n=9$ and  $G$ is  $\psl(2,8)$ or $\pgaml(2,8)$;
    \item  $G$ is $\{1\}$, $\alt$ or $\sym$.
\end{enumerate}
\end{enumerate}
\end{thm}

 \section{Main result}

 The goal of this section is to prove Theorem \ref{main} for all groups of degree at least $10$. This proof is carried out in a sequence of lemmas. The groups of degree less than $10$ will be handled  in the  next section. The results of this section hold for all $n$ unless otherwise stated.

 If $G$ is trivial, then $G$  is obviously normalizing, so we always assume that $G$ is non-trivial.

We start by stating an easy lemma whose proof is self-evident, and  that will be used without further mention. A subset $X$ of $[n]$ is said to be a \emph{section} of a partition $\mathcal{P}$ of $[n]$ if $X$ contains precisely one element in every class of $\mathcal{P}$. The \emph{kernel} of $a\in \trans$ is the equivalence relation $\ker(a)=\{(x,y)\in [n]:(x)a=(y)a\}$. 

\begin{lem}\label{tiny}
Let $G$ be a subgroup of $\sym$ and let $a\in\trans\setminus \sym$. Then, if for some $g,h\in G$ we have $\rank(h^{-1}ahg^{-1}ag\ldots)=\rank(a)$, then exists $h_{1}:=hg^{{-1}}\in G$ such that $h_{1}$ maps the image of $a$ to a section of the kernel of $a$.
\end{lem}

The following lemma is probably well-known: it is an easy generalization of
a result of Birch~\emph{et al.} \cite{birch}.

\begin{lem}\label{enormous}
Let $G$ be a transitive permutation group on $X$, where $|X|=n$. Let $A$ and
$B$ be subsets of $X$ with $|A|=a$ and $|B|=b$. Then the average value of
$|Ag\cap B|$, for $g\in G$, is $ab/n$. In particular, if $|Ag\cap B|=c$ for
all $g\in G$, then $c=ab/n$.
\end{lem}

\begin{proof}
Count triples $(x,y,g)$ with $x\in A$, $y\in B$, and $xg=y$. There are $a$
choices for $x$ and $b$ choices for $y$, and then $|G|/n$ choices for $g$.
Choosing $g$ first, there are $|Ag\cap B|$ choices for $(x,y)$ for each $g$.
The result follows.
\end{proof}

\begin{lem}
Let $G \le S_n$ be normalizing and non-trivial. Then
\begin{itemize}
 \item[(i)] $G$ is transitive;
  \item[(ii)] $G$ is primitive.
\end{itemize}
\end{lem}
\begin{proof}
Regarding (i),  let $A$ be an orbit of $G$ which is not a single point, and suppose
that $|A| < n$. Let $a$ be an (idempotent) map which acts as the
identity on $A$ and maps the points outside $A$ to points of $A$ in any manner.
Then $a$ fixes $A$ pointwise, and hence so does any $G$-conjugate of $a$, and
so does any product of $G$-conjugates: that is, $\langle a^G\rangle$
fixes $A$ pointwise. On
the other hand, if $g\in G$ acts non-trivially on $A$, then so does $ag$, and
$ag\in\langle a,G\rangle\setminus G$. So these two semigroups are not equal,
and $G$ is not normalizing.

Regarding (ii) suppose that $G$ is imprimitive and let $B$ be a non-trivial $G$-invariant partition of $\{1,\ldots,n\}$.
Choose a set $S$ of representatives for the $B$-classes, and let $a$ be the
map which takes every point to the unique point of $S$ in the same $B$-class.
Then $a$ fixes all $B$-classes (in the sense that it maps any $B$-class into
itself), and hence so does any $G$-conjugate of $a$, and so does any product
of $G$-conjugates. On the other hand, the transitivity of $G$  implies that there exists  $g\in G$ that does not fix all $B$-classes,
so that  neither does the element $ag\in\langle a,G\rangle\setminus G$.
As before, it follows that $G$ is not normalizing.
\end{proof}

Now we are ready to prove the main lemma of this section. But before that we introduce some terminology and results. For natural numbers $i,j\leq n$ with $i\leq j$, a  group $G\leq \sym$ is said to be $(i,j)$-homogeneous if for every $i$-set $I$ contained in $[n]$ and for every $j$-set $J$ contained in $[n]$, there exists $g\in G$ such that $Ig\subseteq J$. This notion is linked to homogeneity since an $(i,i)$-homogeneous group is an $i$-homogeneous (or $i$-set transitive) group in the usual sense.

The goal of next lemma is to prove that a normalizing group is $(k-1,k)$-homogeneous, for all
 $k$ such that $1\leq k\leq \lfloor \frac{n+1}{2}\rfloor$. But before stating our next lemma we state here two results about $(k-1,k)$-homogeneous groups. (We denote the dihedral group of order $2p$ by $D(2*p)$.)

 \begin{thm}(See \cite{ArCa12}) \label{thkk-1}
If $n\geq 1$ and $2\leq k\leq \lfloor \frac{n+1}{2} \rfloor$ is fixed, then the following are equivalent:
\begin{enumerate}
    \item[\rm (i)] $G$ is a $(k-1,k)$-homogeneous subgroup of $\sym$;
    \item[\rm (ii)] $G$ is $(k-1)$-homogeneous or $G$ is one of the following groups
    \begin{enumerate}
    \item[\rm (a)] $n=5$ and $G\cong C_5$ or $D(2*5),$  $k=3$;
    \item[\rm (b)]  $n=7$ and $G\cong\agl(1,7)$, with $k=4$;
    \item[\rm (c)]  $n=9$ and $G\cong\asl(2,3)$ or $\agl(2,3)$, with $k=5$.
\end{enumerate}
\end{enumerate}
\end{thm}

These groups admit an  analogue of the Livingstone--Wagner \cite{lw} result about homogeneous groups.

\begin{cor}(See \cite{ArCa12}) \label{corkk-1}
Let $n\geq 1$, let $3\leq k\leq \lfloor \frac{n+1}{2}\rfloor$ be fixed, and let $G\leq \sym$ be a $(k-1,k)$-homogeneous group. Then $G$ is a $(k-2,k-1)$-homogeneous group, except when  $n=9$ and $G\cong\asl(2,3)$ or $\agl(2,3)$, with $k=5$.
\end{cor}

 Now we state and prove the main lemma in this section.

\begin{lem}
Let $G\leq \sym$ be a normalizing group such that $n\geq 10$. Then, for all
 $k$ such that $2\leq  k\leq \lfloor \frac{n+1}{2}\rfloor$, the group $G$ is $(k-1,k)$-homogeneous.
\end{lem}
\begin{proof}
 Suppose that $G$ fails to have the $(k-1,k)$-homogenous property, for some $k< \lfloor\frac{n+1}{2}\rfloor$.
Then it follows that $G$ fails to be $(m-1,m)$-homogeneous, for $m= \lfloor\frac{n+1}{2}\rfloor$, that is, there exist two sets, $I$ and $J$, such that $Ig\not\subseteq J$, for all $g\in G$. Without loss of generality (since we can replace $G$ by some appropriate $g^{{-1}}Gg\leq \sym$) we can assume that $I=\{1,\ldots,m-1\}$,  $J=\{a_{1},\ldots,a_{m}\}$ and hence there is no $g\in G $ such that
$$
\{1,\ldots ,m-1\}g\subseteq \{a_{1},\ldots, a_{m}\}.
$$

Now pick $a\in \trans$ such that
$$a=\left(\begin{array}{ccccccc}
	\{1\}&\ldots &\{m-1\}    &[n]\setminus\{1,\ldots,m-1\}\\
	  a_{1}   &\ldots&a_{m-1}  &a_{m}
\end{array}\right).$$

Observe that (for all $g\in G$) we have  $\rank(aga)<\rank(a)$, because there is no set in the orbit of $\{a_{1},\ldots,a_{m}\}$ that contains $\{1,\dots ,m-1\}$; therefore there is only one chance for $G$ to normalize $a$:
 \begin{eqnarray}\label{obs}
 (\forall g \in G)(\exists h\in G)\ ag = h^{-1}ah.
 \end{eqnarray}

 On the other hand,
 $$
 |\{a_{1},\ldots,a_{m}\}\cap \{1,\dots , m-1\}|=r,
 $$
 implies that  $\rank(a^{2})=r+1$, and hence  $\rank((h^{-1}ah)^{2})=r+1$ as well.

 Now we have two situations: either there exists a constant $c$ such that for all $g\in G$ we have
 $$
 |\{a_{1},\ldots,a_{m}\}g\cap \{1,\dots , m-1\}|=c,
 $$ or not.

 We start by the second case. We are going to build  a map $ah\in \trans\setminus \sym$ and pick a permutation $h^{-1}g\in G$ such that $(ah)h^{-1}g$ is not normalized by $G$.

By assumption there exists $g\in G$ such that
 $$
 |\{a_{1},\ldots,a_{m}\}g\cap \{1,\ldots ,m-1\}|=c
 $$ and there exists $h\in G$ such that
 $$
 |\{a_{1},\ldots,a_{m}\}h\cap \{1,\ldots ,m-1\}|=d<c.
 $$

 Then, by the observation above, $\rank((ah)^{2})=d+1$ and so the rank of any one of its conjugates is also $d+1$: for all $h_{1 }\in G$ we have $\rank((h^{-1}_{1}(ah)h_{1})^{2})=d+1$.

 On the other hand, $\rank((ah\cdot h^{-1}g)^{2})=c+1(>d+1)$ so that
 $$
(\forall h_{1}\in G) ah\cdot h^{-1}g \neq h^{-1}_{1}(ah)h_{1}
  $$ and hence by (\ref{obs})
 $$
 ah\cdot h^{-1}g \not\in \langle (ah)^{h_{1}}\mid h_{1}\in G\rangle ,
  $$
  a contradiction. It is proved that if the size of the following intersection
  $$
 |\{a_{1},\ldots,a_{m}\}g\cap \{1,\dots , m-1\}|
 $$
  varies with $g\in G$, then it is possible to build a map that is not normalized by $G$.

  Now we turn to the first possibility, namely, exists a constant $c$ such that, for all $g\in G$, we have
  $$
 |\{a_{1},\ldots,a_{m}\}g\cap \{1,\dots , m-1\}|=c.
 $$

First observe that if  $c=1$, then  $m(m-1)=n$, which holds only when $n=6$ (see Lemma \ref{enormous} and recall that $m=\lfloor\frac{n+1}{2}\rfloor$). Since $n\geq 10$ we have  $c\geq 2$.

As $
 |\{a_{1},\ldots,a_{m}\}g\cap \{1,\dots , m-1\}|=c
 $, for all $g\in G$, it follows that (for $g=1$) we have
 $|\{a_{1},\ldots,a_{m}\}\cap \{1,\dots , m-1\}|=c.$
Without loss of generality (in order  to increase the readability of the map $a$ below), we will assume that $a_{i}=i$, for $i=1,\ldots ,c$.

Now, as $G$ is transitive, pick $g\in G$ such that $1g=2$, and suppose there exists $h\in G$ such that $ag=a^{h}$, with

$$a=\left(\begin{array}{ccccccc}
	\{1\}&\ldots &\{c\}&\{c+1\}&\ldots &\{m-1\}  &[n]\setminus\{1,\ldots,m-1\}\\
	  1   &\ldots &c&a_{c+1} &\ldots  &a_{m-1}&a_{m}
\end{array}\right),$$
$$ag=\left(\begin{array}{ccccccc}
	\{1\}&\ldots &\{c\}&\{c+1\}&\ldots &\{m-1\}  &[n]\setminus\{1,\ldots,m-1\}\\
	  1g=2   &\ldots &cg&a_{c+1}g &\ldots  &a_{m-1}g&a_{m}g
\end{array}\right)$$
and
$$a^{h}=\left(\begin{array}{ccccccc}
	\{1\}h&\ldots &\{c\}h&\{c+1\}h&\ldots &\{m-1\}h  &[n]\setminus\{1,\ldots,m-1\}h\\
	  1h   &\ldots &ch&a_{c+1}h &\ldots  &a_{m-1}h&a_{m}h
\end{array}\right).$$

In $ag$, $2$ is not a fixed point and $|2(ag)^{{-1}}|=1$. Therefore $2$ is not a fixed point of $a^{h}$ and $|2(a^{h})^{{-1}}|=1$. As the possible non-fixed points of $a^{h}$ with singleton inverse image (under $a^{h}$) are contained in  $\{a_{c+1}h, \ldots , a_{m-1}h\}$, it follows there must be an element $a_{j}\in \{a_{c+1}, \ldots  ,a_{m-1}\}$ such that $a_{j}h=2$.   But this means that $h$ does not permute $\{1,\ldots,m-1\}$ and hence
$$\{\{1\},\dots,\{m-1\}\}h\neq \{\{1\},\ldots, \{m-1\}\}$$
yielding that the kernel of $a^{h}$ and $ag$ are different, a contradiction.

  It is proved that if $G$ fails to be $(k-1,k)$-homogeneous, for some $k$ such that $1\leq k\leq \lfloor \frac{n+1}{2} \rfloor$, then $G$ is not normalizing. The result follows.
  \end{proof}

We have now everything needed in order to prove Theorem \ref{main} regarding the groups of degree at least $10$. In fact, if  $G$ is normalizing, then $G$ is $(k-1,k)$-homogenous for all $k$ such that $1< k\leq \lfloor \frac{n+1}{2} \rfloor$ and hence the group (of degree at least $10$) is $(k-1)$-homogeneous (by Theorem \ref{thkk-1}). A primitive group (of degree $n$)  is proper if it does not contain the alternating group of degree $n$. Therefore, if $n=10$, then a proper primitive normalizing  group must be  $(k=\lfloor \frac{n-1}{2} \rfloor=4)$-homogenous, but there are no such groups of degree $10$. For $n=11$, a proper primitive normalizing  group must be  $(k=\lfloor \frac{n-1}{2} \rfloor=5)$-homogenous, but there are no such groups of degree $11$. If  $n=12$, then the group must be  $(k=\lfloor\frac{n-1}{2} \rfloor=5)$-homogenous, whose unique example (of degree $12$) is $M_{12}$. However $M_{12}$, as the group of permutations of $\{1,\ldots ,12\}$ generated by the following permutations
\[\begin{array}{lll}
(1\ 2\ 3)(4\ 5\ 6)(7\ 8\ 9),&
 (2\ 4\ 3\ 7)(5\ 6\ 9\ 8),& 
  (2\ 9\ 3\ 5)(4\ 6\ 7\ 8),\ \\
  (1\ 10)(4\ 7)(5\ 6)(8\ 9),&
   (4\ 8)(5\ 9)(6\ 7)(10\ 11),&
    (4\ 7)(5\ 8)(6\ 9)(11\ 12),
    \end{array}
\]
 fails to normalize the following map:

 $$a=\left(\begin{array}{cccccc}
	\{1\}&\{2\}&\{3\}&\{4\}&\{5,6\}&\{7,\ldots,12\}\\
	  1   &2&3&4&5&6
\end{array}\right).$$

 In fact, it is easily checked (using GAP \cite{GAP}) that no element of $M_{12}$ maps $\{1,\ldots,6\}$ to a
section for the kernel of this map $a$. So, by Lemma \ref{tiny}, we only have to
check whether, for every $g\in M_{12}$, there exists $h\in M_{12}$ such that
$ag=h^{-1}ah$. This fails for $g=(132)(465)(798)$.

 For $n>12$, the group must be $(k=\lfloor\frac{n-1}{2} \rfloor\geq 6)$-homogenous, but for $k\geq 6$ there are no proper primitive $k$-homogeneous groups   \cite[Theorem 9.4B, p. 289]{dixon}.

Therefore the unique groups that can be normalizing are the trivial group, the symmetric and alternating groups, and some primitive groups of degree at most $9$. In the next section we explain how we used {\sf GAP} \cite{GAP}, {\sf orb} \cite{orb} and {\sf Citrus} \cite{citrus}, to check these groups of small degree. That the symmetric and the alternating groups are normalizing is already well known.

\begin{thm}(\cite[Theorem 5.2]{lmm})\label{sym}
The groups $\sym$ and $\alt$ are normalizing.
\end{thm}

\section{Computational considerations}

In this section we describe the computational methods used to find the normalizing groups of degree at most $9$.
Regarding primitive groups of degree at most $3$ they contain the alternating group and the result follows by Theorem \ref{sym}. Therefore, from now on we assume that $4\leq n\leq 9$.
We know that a normalizing group $G\leq \mathcal{S}_{n}$ is primitive and
$(k-1,k)$-homogeneous for all $k\leq\lfloor\frac{n+1}{2}\rfloor$. By
Theorem \ref{thkk-1} we have two situations:

\begin{enumerate}
\item  $G$ is $(\lfloor \frac{n-1}{2}\rfloor)$-homogeneous and hence (by inspection of the GAP library of primitive groups) is one of the groups below:

$\ $

\begin{center}
\begin{tabular}{|c|l|}
\hline
Degree&$G$\\
\hline \hline
5&$\agl(1,5)$\\
\hline
6&$\psl(2,5)$, $\pgl(2,5)$\\
\hline
8&$\agl(1,8)$, $\agaml(1,8)$, $\asl(3,2)$, $\psl(2,7)$, $\pgl(2,7)$\\
\hline
9&$\psl(2,8)$, $\pgaml(2,8)$ \\
\hline
\end{tabular}
\end{center}

$\ $

\item or $G$ is one of the groups in Theorem \ref{thkk-1} ($C_{5}$ and $D(2*5)$ of degree $5$; $\agl(1,7)$ of degree $7$; $\asl(2,3)$ and $\agl(2,3)$ of degree 9).
\end{enumerate}

%
%
%
%
%
%
%

To check that a group $G \le \sym$ is
$a$-normalizing for some $a \in \trans\setminus\sym$ it is enough to check
that $aG \subseteq \langle g^{-1}ag \mid g\in G\rangle$, since
the latter is closed under conjugation with elements from $G$.
So we only have to enumerate the $G$-orbit of $a$ with right multiplication
as action and check  membership in the semigroup
$\langle g^{-1}ag \mid g \in G\rangle$ for all its elements. This is essentially achieved
by the following \textsf{GAP}-commands using the packages \textsf{orb}
(see \cite{orb}) and \textsf{Citrus} (see \cite{citrus}):

\begin{verbatim}
gap> o := Orb(G,a,OnRight);; Enumerate(o);;
gap> o2 := Orb(G,a,OnPoints);; Enumerate(o2);;
gap> s := Semigroup(o2);;
gap> ForAll(o,x->x in s);
true
\end{verbatim}

However, for the larger examples on $9$ points checking this for all $a \in
\trans \setminus \sym$ would have taken too long. Fortunately, this was not
necessary, since if $G$ is $a$-normalizing, then it is of course
$a^g$-normalizing for all $g \in G$. So we only have to check this
property for representatives of the $G$-orbits on $\trans\setminus\sym$
under the conjugation action.

To compute a set of representatives
we first implemented an explicit bijection of $\trans$ to the set
$\{ i \in \mathbb{N} \mid 1 \le i \le n^n \}$. Then we organised
a bitmap of length $n^n$ and enumerated all conjugation $G$-orbits
in $\trans$, crossing off the transformations we had already encountered
in the bitmap. Having the representatives as actual transformations
then allowed us to perform the test explained above.

A slight speedup was achieved by actually verifying a stronger condition,
namely that $aG$ is a subset of the $\mathcal{R}$-class of $a$ in the semigroup
$\langle a^g \mid g \in G\rangle$, which turned out to be the case whenever
$G$ was normalizing. Testing membership in the $\mathcal{R}$-class of $a$ in the transformation semigroup
$S:=\langle a^g \mid g \in G\rangle$ can be done
by computing the strong orbit of the image of $a$ under the action of $S$
and the permutation group  induced by the elements of $S$ that stabilise
the image of $a$ setwise; as described in \cite{Linton1998aa}. This method
is implemented in the {\sf Citrus} package \cite{citrus} for {\sf GAP}.

For degree $5$, only $\agl(1,5)$ is normalizing, since the group $C_{5}$
fails to normalize the map

\[ a=\left(\begin{array}{ccc}
	\{1,2,5\}&\{3\}&\{4\}\\
	  1   &3&4
\end{array}\right), \]
\noindent and the group $D(2*5)$ fails to normalize the map
\[ a=\left(\begin{array}{ccc}
	\{1,2,3\}&\{4\}&\{5\}\\
	  1   &3&2
\end{array}\right). \]

For degree $6$, both groups $\psl(2,5)$ and $\pgl(2,5)$ are normalizing.

For degree $7$, we only had to check $\agl(1,7)$, which fails to normalize
the map
\[ a=\left(\begin{array}{ccc}
	\{1,\ldots,5\}&\{6\}&\{7\}\\
	  1   &2&3
\end{array}\right). \]

For degree $8$, all three groups $\agl(1,8)$, $\agaml(1,8)$ and
$\asl(3,2)$ fail to
normalize the map
\[ a=\left(\begin{array}{cccc}
	\{1,\ldots,5\}&\{6\}&\{7\}&\{8\}\\
	  1   &2&3&4
\end{array}\right), \]
the group $\psl(2,7)$ fails to normalize the map
\[ a=\left(\begin{array}{cccc}
	\{1,\ldots,5\}&\{6\}&\{7\}&\{8\}\\
	  1   &2&3&5
\end{array}\right), \]
and finally the group $\pgl(2,7)$ fails to normalize the map
\[ a=\left(\begin{array}{cccc}
	\{1,\ldots,5\}&\{6\}&\{7\}&\{8\}\\
	  1   &2&4&7
\end{array}\right). \]

For degree $9$, the two groups $\psl(2,8)$ and $\pgaml(2,8)$ are
normalizing, whereas both groups $\asl(2,3)$ and $\asl(2,3)$ fail to
normalize the map
\[ a=\left(\begin{array}{cccccc}
	\{1,8\}&\{2,3,7\}&\{4\}&\{5\}&\{6\}&\{9\}\\
	  7   &8&6&9&4&5
\end{array}\right). \]

These computational results complete the proof of our main Theorem
\ref{main}.

\section{Problems}

Regarding this paper, the main problem that has to be tackled now should be the classification of the $k$-normalizing groups.

\begin{problem}
Let $k$ be a fixed number such that $1<k<\lfloor \frac{n+1}{2}\rfloor$. Classify the $k$-normalizing groups, that is, classify the groups that satisfy $\langle a,G\rangle\setminus G=\langle a^{g}\mid g\in G\rangle$,   for every rank $k$ map.
\end{problem}	

To solve this problem is necessary to use the results of \cite{ArCa12}, but that will be just a starting point since many delicate considerations will certainly be required.

The theorems and problems in this paper admit linear versions that are interesting for experts in groups and semigroups, but also to experts in linear algebra and matrix theory. For the linear case, we already know that any singular matrix with any group containing the  special linear group is normalizing \cite{ArSi1,ArSi2} (see also the related papers \cite{Gr,Pa,Ra}).

\begin{problem}
Classify the linear groups $G\leq GL(n,q)$ that, together with any singular linear transformation $a$, satisfy
$$
\langle a,G\rangle \setminus G = \langle h^{-1}a{h}\mid h\in G\rangle.
$$
\end{problem}

A necessary step to solve the previous problem is to solve the following.
\begin{problem}\label{11}
Classify the groups $G\leq GL(n,q)$  such that for all rank $k$ (for a given $k$) singular matrix $a$ we have that $\rank(aga)=\rank(a)$, for some $g\in G$.
\end{problem}

To handle this problem it is useful to keep in mind the following results. Kantor~\cite{kantor:inc} proved that if a subgroup of $\pgaml(d,q)$ acts transitively on $k$-dimensional subspaces, then it acts transitively on $l$-dimensional subspaces for all $l\le k$ such that $k+l\le n$; in~\cite{kantor:line}, he showed that subgroups transitive on $2$-dimensional subspaces are $2$-transitive on the $1$-dimensional subspaces with the single exception of a subgroup of $\pgl(5,2)$ of order $31\cdot5$; and, with the second author~\cite{cameron-kantor}, he showed that such groups must contain $\psl(d,q)$ with the single exception of the alternating group $A_7$ inside $\pgl(4,2)\cong A_8$. Also Hering \cite{He74,He85} and Liebeck \cite{Li86} classified the subgroups of $\pgl(d,p)$ which are transitive on $1$-spaces. (See also \cite{kantor:inc,kantor:line}.)

\begin{problem}
Solve analogues of the results (and problems) in this paper  for independence algebras (for definitions and fundamental results see \cite{ArEdGi,arfo,cameronSz,gould}).
\end{problem}

\section*{Acknowledgements}
The authors would like to express their gratitude to the referee for a very careful review and for suggestions that prompted a much simplified paper.
 
The first author was partially supported  by FCT through the following projects: PEst-OE/MAT/UI1043/2011, Strategic Project of Centro de \'Algebra da Universidade de Lisboa; and PTDC/MAT/101993/2008, Project Computations in groups and semigroups .

The second author is grateful to the Center of Algebra of the University of Lisbon
for supporting a visit to the Centre in which some of this research was done.

\end{document}